\documentclass{article}
\usepackage{amssymb,latexsym, amsmath,amscd,xspace,ifthen}
\usepackage{amsthm,enumerate}

\usepackage{tabularx}
\usepackage{mathabx}
\newtheorem{thm}{Theorem}
\newtheorem{defn}[thm]{Definition}
\newtheorem{lemma}[thm]{Lemma}
\newtheorem{conj}[thm]{Conjecture}
\newtheorem{cor}[thm]{Corollary}
\newtheorem{rem}[thm]{Remark}
\newcommand{\one}{\underline{\bold{1}}}
\newcommand{\mynull}{ \tiny $\times$ \normalsize }

\begin{document}




\title{New Upper Bounds on the Minimal Domination Numbers of High-Dimensional Hypercubes}


\author{Zachary DeVivo\footnote{University of Michigan, zdevivo@umich.edu} \, and Robert K. Hladky\footnote{American Mathematical Society, Mathematical Reviews, rkh@ams.org, corresponding author} }

            


\maketitle

\begin{abstract} We briefly review known results on upper bounds for the minimal domination number $\gamma_n$ of a hypercube of dimension $n$, then present a new method for constructing dominating sets. 

Write $n =2^{\widehat{n}}-1 +{\widecheck{n}}$ with $0\leq {\widecheck{n}}<2^{\widehat{n}}$. Our construction applies to all $n$ lying within the expanding wedge $\theta({\widehat{n}}) \leq {\widecheck{n}} < 2^{{\widehat{n}}}$, where $\theta$ is a specific, easily computable function with the asymptotic property $\theta(a) \sim 2^{a/2}$.  For all $n$ within the smaller wedge  $\theta({\widehat{n}}) \leq {\widecheck{n}} < 2^{{\widehat{n}}-2}$, the resulting upper bound on $\gamma_n$ betters those previously known.

\vspace{2pt}
\noindent Keywords: \textit{Minimal Domination Number,  Hypercubes,  Binary Codes}

\noindent MSC 2020: \textit{05C69,  94B05}
\end{abstract}





\section{Background and Previous Bounds}

Trying to determine the optimal domination number for hypercubes is a surprisingly formidable task with several intriguing properties: (i) the problem is simple to state and accessible to those without extensive mathematical training; (ii) there is still scope for considerable progress; (iii) even the known solutions in the low-dimensional cases are somewhat non-intuitive. If you've never done so before, constructing minimal dominating sets for $n=5,6$ can be done by hand, but is a more challenging task than one might expect. Solutions for the families $n = 2^{\widehat{n}} -1$ and $n = 2^{\widehat{n}}$ are well known and it has been proved that $\gamma_9= 62$.  Surprisingly, the problem is open for $n\geq 10$ outside of the above two families.  Even for $n=9$, there is no satisfying or intuitive construction;  an explicit minimal dominating set of length $62$ was found via simulated annealing in \cite{Wille96}, and was proven optimal in \cite{MR1873942}.

The $n$-dimensional hypercube $\mathcal{Q}_n$ has multiple equivalent definitions arising from its appearance in various sub-disciplines of mathematics. For instance, it can be described as
(i) the $n$-dimensional vector space over the field $\mathbb{F}_2$, equipped with the standard Euclidean norm;
(ii)  the metric space of binary words of length $n$, equipped with the Hamming distance, i.e., $d(x,y) =\left| \{ i \big| x_i \ne y_i \}\right|$;
(iii) the power set $2^{\{1,\dots,n\}}$ with distance determined by the the size of the symmetric difference, i.e.,   $d(A,B) = \left|  A \ominus B\right|$;
(iv) the recursively defined graph $\mathcal{Q}_n = K_2 \square \mathcal{Q}_{n-1}$, where $K_2$ is the graph with $2$ vertices connected by a single edge and $\square$ is the Cartesian product of graphs.

A subset $\Delta \subset \mathcal{Q}_n$ is said to be \textit{dominating of degree} $r$ if 
\[ \mathcal{D}_r (\Delta) \coloneq \{  y \in \mathcal{Q}_n \big| \exists x \in \Delta  \text{ such that } \|x-y\|)\leq r\} = \mathcal{Q}_n.\]
In the computer science literature, which predominantly uses the binary word formulation, dominating sets are called a binary codes with covering radius $r$. The \textit{minimal domination number} $\gamma^{(r)}_n$ is defined to be the minimum of $\left| \Delta \right|$ over all dominating sets $\Delta$. A dominating set of size $\gamma^{(r)}_n$ is said to be \textit{minimal}.
\par
 In this paper, we shall solely be concerned with the case $r=1$, write $\mathcal{D}$ for  $\mathcal{D}_1$ and $\gamma_n$  for $\gamma_n^{(1)}$, respectively. Even in this case, the precise value of $\gamma_n$ is only known for $n \leq 9$ and for the families $n = 2^k, 2^k-1$.  However, the bounds 
\[ 2^{n- {\widehat{n}} }  \geq   \gamma_n  \geq  \frac{2^n}{n+ \one(\text{$n$ odd})} =  2^n \cdot \begin{cases} 1/n, \quad & \text{$n$ even,}\\ 1/(n+1), \quad  & \text{$n$ odd,} \end{cases} \]
have been known since \cite{MR945316}, where the even (and significantly harder) case was proved. Here, we are using the notation
\begin{equation}  n = 2^{\widehat{n}} -1 + {\widecheck{n}}, \quad 0 \leq {\widecheck{n}} < 2^{\widehat{n}},
\end{equation}
and $\one$ represents the binary truth function, i.e. it takes the value $1$ if its input is true and $0$ otherwise. When ${\widecheck{n}} = 0,1$, i.e., $n =2^{\widehat{n}}-1$ or $2^{\widehat{n}}$, the lower and upper bounds are easily seen to be equal and hence sharp. 
\par
The additional properties (i) $\gamma_n \leq \gamma_{n+m} \leq 2^m  \gamma_n$, and  (ii) $\gamma_{2n+1} \leq 2^n \gamma_n$ are also well-known.  (i) is trivial.  (ii), which dates back at least to \cite{mallard}, requires a little cleverness. Because the details of the proof,  not just the result, will be important later, we state it as a lemma.  Here, we use the vector space formulation of $\mathcal{Q}_n$, and so all addition operators must be interpreted in the $\mathbb{F}_2$ context. 

\begin{lemma}\label{propagate}
For all positive integers $n$, $\gamma_{2n+1} \leq 2^n  \gamma_n$.
\end{lemma}
\begin{proof} Let $\Delta$ be a dominating set for $\mathcal{Q}_n$ and  $\pi \colon \mathcal{Q}_n \to \mathbb{F}_2$ be the $\mathbb{F}_2$-linear map $\pi(x_1,\dots,x_n) = \sum x_i$. Now set  \[ \widetilde{\Delta} = \{  ( x, x+ \sigma , \pi(x)) \big| x \in \mathcal{Q}_n, \sigma \in \Delta\}  \subset \mathcal{Q}_n \times \mathcal{Q}_n \times \mathcal{Q}_1 \cong \mathcal{Q}_{2n+1}. \]
Pick any  $(y,z,\epsilon) \in \mathcal{Q}_n \times \mathcal{Q}_n \times \mathcal{Q}_1$. By definition, there are $\sigma \in \Delta$ and $\delta \in \mathcal{Q}_n$ with $\|\delta\| \leq 1$ such that $z-y = \sigma +\delta$. Then it is easy to see that one of  \[ ( y ,  y+\sigma, \pi(y)) = ( y,  z -\delta, \pi(y)), \quad  (y+ \delta,   y+\delta+\sigma,, \pi(y+\delta)) =(y+\delta, z, \pi(y)+1),\] which are both contained in $\widetilde{\Delta}$, must lie within distance one of $(y,z,\epsilon)$. Noting that $\big|\widetilde{\Delta} \big| = 2^n \left|\Delta\right|$ completes the proof.
\end{proof} 

\begin{rem}
The above construction preserves many additional properties that appear elsewhere in the literature, e.g. \textit{independence}, where the sets considered are not permitted to contain adjacent vertices. Hence $\gamma^i_{2n+1} \leq 2^n \gamma^i_n$ also. Most of the results of this section, go through in these cases without much modification.
\end{rem}

One immediate consequence is the existence of minimal dominating sets with special structure in the case $\widecheck{n}=0$. This elementary observation will actually underpin our new bounds in later sections.
\begin{cor}\label{graph}
 For every $\widehat{n} >1$, there is an $\mathbb{F}_2$-linear map $h \colon \mathcal{Q}_{2^{\widehat{n}} -1 -\widehat{n}} \to \mathcal{Q}_{\widehat{n}}$ such that the graph $\Gamma = \{ (x,h(x)) \big| x \in \mathcal{Q}_{2^{\widehat{n}} -1 -\widehat{n}} \}$ is a minimal dominating set for $ \mathcal{Q}_{2^{\widehat{n}} -1}$.
\end{cor}

The proof  is a simple exercise in iterating the construction of Lemma \ref{propagate}, beginning with the elementary dominating set $\{0\}$ for $\mathcal{Q}_1$, and at every stage permuting the terms involving $\pi$ to the end. Finding an exact expression for $h$ would require elementary but irritating linear algebra, but the existence of such an $h$ is clear.

The true power of Lemma \ref{propagate} becomes more evident if we make two more definitions. First, we define the \textit{multiplicative gain in dimension $n$} by 
\begin{equation}
\chi_n :=1 - \gamma_n \cdot 2^{ \widehat{n} -n } \qquad \text{(and so $\gamma_n = (1-\chi_n) 2^{ n- \widehat{n} }$ )}.
\end{equation}
Second, for any positive integer, the \textit{domination wedge} with vertex $n$ is defined to be
\[ W(n) := \left\{ p \in \mathbb{Z} \Big|    (n+1)  \leq \frac{ p+1}{2^l}  < 2^{{\widehat{n}}+1 }   \text{ for some $l\geq0$}\right\}.\]
The statement $p \in W(n)$ is easily seen to be equivalent to $ \widehat{p} = {\widehat{n}} + l$  and $\widecheck{p} \geq 2^l \cdot {\widecheck{n}} $ for some $l \geq 0$.  All known results for $\gamma$ can easily be restated in terms of $\chi$. In particular Lemma \ref{propagate} can then be restated as the third part of following theorem, which simply lists known properties of $\chi$. 

\begin{thm}\label{wedge}\hfill
\begin{enumerate}[(1)]
\item $\chi_n =0$ if ${\widecheck{n}} =0,1$,
\item $ 0 \leq \chi_n  \leq  1 -  \frac{2^{\widehat{n}}}{n+ \one(\text{$n$ odd})}  < \frac{1}{2}$,
\item If $ m \in W(n)$ then $\chi_m \geq \chi_n$.
\end{enumerate}
\end{thm}

The best known (upper and lower) bounds\footnote{At least, as of November 2011. But the authors are unaware of any improvements.} on $\gamma_n$ for $n \leq 33$ (and $r \leq 10$) have been collected in \cite{Keri}. The same values were listed in \cite{MR4658424}. In Figure \ref{chi1},  we have converted these values to the best known lower bounds on $\chi_n$, which permits easier comparison. We remark that \cite{Keri} includes a variety of citations for upper bounds on $\gamma_n$ in the range $19 \leq n \leq  29$. However, they all follow immediately from Theorem \ref{wedge} and bounds computed for smaller $n$. Another surprising feature of the problem is that (other than Lemma \ref{propagate}) there is little in the way of a general theory. Successful attacks on the problem have largely been limited to a single dimension. For $n=9,10$ and $12$ the best results were achieved by simulated annealing, a flow-based numerical technique borrowed from engineering. Bucking the trend, the best known upper bound for $n=11$ is based on an elegant construction. However, it relies on the existence of a particular Steiner system. Since these are rare, the methodology doesn't apply to other dimensions. In $n=13$, a  group theoretic techniques were used to reduce the computational complexity to the point where a stochastic optimization method known as Tabu search could be applied.

 \begin{figure}\label{chi1}
\caption{Best known upper bounds on $\chi$. Here  $^*$ denotes the bound is known to be sharp. The note field either indicates a citation or gives the vertex of the domination wedge the bound results from. A question mark means that $\chi>0$  is expected, but unverified.}
\bgroup
\small
\setlength{\tabcolsep}{0.4em}
\begin{center}
\def\arraystretch{1.3}
\begin{tabular}{c|ccccccccccccccccc}
$n$ &  $2^*$ & $3^*$ & $4^*$ & $5^*$ & $6^*$ & $7^*$ & $8^*$ & $9^*$ & 10 & 11 & 12 & 13 &14 & $15^*$ &$16^*$ & 17 \\
\hline
$\chi_n$ &$0$  &$0$ &$0$  &  $\frac{1}{8}$ &  $ \frac{1}{4}$ &$0$ &$0$  & $\frac{1}{32}$ & $\frac{1}{16}$ & $ \frac{1}{4}$ & $\frac{33}{128}$ & $\frac{5}{16}$ & $\frac{5}{16}$ & 0 & 0 & 0 \\
\text{note} & & & & &  & & & \scriptsize \cite{Wille96} &  \scriptsize \cite{MR1041836} &  \scriptsize\cite{MR859092}& \scriptsize \cite{Wille96} & \scriptsize \cite{MR1669523} & 13 & & & ?\\
\multicolumn{2}{c}{}\\
$n$  &  18& 19& 20 & 21 & 22 & 23 & 24 & 25 & 26 & 27  & 28 & 29 & 30 &$31^*$ & $32^*$ & 33\\
\hline
$\chi_n$& 0 &  $\frac{1}{32}$ & $\frac{1}{32}$  & $ \frac{1}{16}$ &$ \frac{1}{16}$  &$\frac{1}{4} $ &$\frac{1}{4} $  & $\frac{33}{128}$& $ \frac{33}{128}$ & $\frac{5}{16}$ & $\frac{5}{16}$  & $\frac{5}{16}$  & $\frac{5}{16}$ & 0 & 0  &0\\
\text{note}  & ?   & 9 & 9 &  10 & 10 & 11 & 11 & 12 & 12 & 13  &  $13$ &  $13$ &  $13$ & & & ?\\
\end{tabular}
\end{center}
\egroup
\end{figure}

It is surprising that little progress has been made since 1999! It may be that the majority of these bounds are in fact sharp. However, it would be somewhat shocking if turned out that $\chi_{14} = \chi_{13}$.  Additionally, every genuinely new bound in this table occurs with ${\widehat{n}} \leq 3$.  It seems likely that further improvements are achievable at the ${\widehat{n}} =4$ level, i.e. for $ 17 \leq n \leq 30$.

There are a handful of known improvements in specific high dimensions. In \cite{MR1483747}, a method based in linear algebra was used to show that $\gamma_{\gamma_n -1} \leq \gamma_n \cdot 2^{\gamma_n - n-1}$.  This was applied to $n=9$ to find the improved bound $\gamma_{61} \leq 5 \cdot 2^{53}$, or equivalently $\chi_{61} \geq 3/8$ (which also implies $\chi_{62} \geq 3/8)$ .  However this approach  quickly falters both because we don't know the precise value of $\gamma_n$ for $n >9$ outside of the cases $\widecheck{n}=0,1$ (for which it yields no new information) and because $\gamma_n$ grows rapidly in $n$. However, in that same paper, more refined techniques were used to show that $\chi_{46} \geq \frac{1}{8}$ and, combined with work from \cite{MR290860}, that $\chi_{56}  \geq \frac{711}{2048}$.

The ad hoc nature of these ingenious solutions means it is hard to make general predictions. All values of $n$ for which $\chi_n$ was previously known to be positive lay inside $W(9)$. In the next two sections, we shall present a new method that provides new and improved upper bounds on an infinite set of $n$. While these new bounds are almost certainly not sharp, they apply in many cases untouched by previous theory. In particular, we obtain results for many $n \notin W(9)$ and, indeed, the range of indices for which we find bound improvements is not contained in $W(q)$ for any $q>1$.

We end this introductory section with the following two-part conjecture:

\begin{conj}\label{conj}\hfill
\begin{enumerate}[(1)]
\item $\chi_n$ is strictly increasing on each interval $2^m \leq n \leq 2^{m+1}-2$
\item $\limsup\limits_{n \to \infty} \chi_n  =  \frac{1}{2}$.
\end{enumerate}
\end{conj}

To the best of the authors' knowledge, we are nowhere near a proof of either part of Conjecture \ref{conj}. For Part (1) a quick glance at Figure \ref{best} indicates how much work must be done to achieve a positive proof. For instance, the best known bounds give $\chi_{13}, \chi_{14} \geq 5/16$ with neither bound known to be sharp.  For Part (2), insufficiently many cases have been computed to predict patterns with any confidence. However the conjecture still seems likely. It would be surprising and very interesting if either part turned out to be false. An answer in the negative to the second part would imply the existence of some obstruction to $\mathcal{Q}_{2^n-2}$ approaching the structural simplicity of $\mathcal{Q}_{2^n-1}$ at very large $n$.


As (fairly weak) evidence in support of Part (2), we note that the two well-known lower bounds can be rewritten as $\chi_n \leq 1- \frac{2^{\widehat{n}}}{n+1}  $ for $n$ odd, and $\chi_n \leq 1- \frac{2^{\widehat{n}}}{n}$ for $n$ even.  By definition $  2^{\widehat{n}} -1\leq n < 2^{{\widehat{n}}+1}-1$  with $n$ increasing from one extreme to the other, then cycling back whenever ${\widehat{n}}$ jumps by $1$. Thus, in both cases, the limit supremum of the bounds also equals $1/2$.

There are two general improvements relevant here: 
\begin{itemize}
\item $\chi_n \leq 1-  \frac{(n-2)}{n-2-2/n} \cdot \frac{2^{\widehat{n}}}{n}$ for $n= 0 \text{ mod } 6$, see  \cite{MR4658424};

\item $\chi_n \leq 1- \left(1 +  \frac{10}{ 5 \binom{n}{2}  -n +2}  \right)  \cdot \frac{2^{\widehat{n}}}{n+1}$ for $n =5 \text{ mod } 6$, $n\geq 11$, see \cite{MR1477282}.
\end{itemize}
However, both these improvements clearly satisfy the same property.

\section{Main Technical Result}

The big takeaway from the last section is that an improved bound for any individual value of $n$ results in improved bounds on the entire domination wedge $W(n)$. This implies that it is beneficial to find bounds for $n$ where ${\widecheck{n}}$ is small. The bounds arising from Figure \ref{chi1}  do provide information for a wide swath of cases, but  the ranges of $n$ to which they apply have the asymptotic property that $ {\widecheck{n}}= \Omega(2^{\widehat{n}})$. 

We shall now work toward providing novel upper bounds that apply in sufficient generality to capture (comparatively) small values of ${\widecheck{n}}$. Throughout this section, we shall use the binary word formulation for hypercubes, i.e. $\mathcal{Q}_n$ will be identified with the set of all binary words of length $n$. In what follows, we'll use the notation  $\delta^p_q$ for the binary word of length $p$ for which all letters are zero except for a single one in the $q$-th position (from the left).  We shall denote concatenation of words either by juxtaposition or by the symbol $\cdot$. The reader is cautioned to recall that $\cdot$ then obeys different algebraic rules than is usually associated with a product. For instance, it obeys the distribution rule 
\[   (a\cdot b) + (c\cdot d)  =  (a+c)\cdot (b+d)=(a \cdot d) + (c \cdot b) .\]

To present our technical result in a workable form, we need to make several definitions. First, for any positive integer $q$ we define $\Theta(q) $ to be the unique pair $(\theta,\xi)$ of non-negative integers such that
\begin{equation}\label{Theta}  \binom{\theta+1}{2} = q +\xi  \quad \text{and} \quad  0 \leq \xi<\theta+1,\end{equation}
i.e., $\theta$ is the smallest integer such that   $1+ 2+\cdots + \theta \geq q$ and   $\xi$ is the excess. For later convenience, we also define $\Theta^{(2)} (q):= (\theta,\xi,\Theta(\xi))$ where $(\theta,\xi) =\Theta(q)$.


Now fix $l>1$ and set $(\theta,\xi)  = \Theta(l+1)$.  For $m \geq 0$, we define an $m$-decomposition of $l$ to be a set $S= \{s_1,\dots,s_q\}$ such that  
\begin{equation}\label{S} 1 \leq s_1 <s_2 < \cdots < s_q < \theta \quad \text{ and } \quad \theta+ m + \sum\limits_s S = l+1.\end{equation}
For simplicity of notation, we set $\Sigma := \sum_s S$. We shall apply such decompositions primarily in the context of binary words. Here an  $m$-decomposition  $S$ of $l$ results in a splitting:
 \begin{equation}\label{decom}
 \underbrace{\star \cdots \star}_l = \underbrace{\underbrace{\star \cdots \star}_{s_1} \underbrace{\star \cdots \star}_{s_2} \underbrace{\star \cdots \star}_{s_q} \underbrace{\star \cdots \star}_{\theta-1}}_{l-m} \underbrace{\star \cdots \star}_{m}.
\end{equation}
For every $s \in S$, we can define projection maps $\pi^{(s)} \colon \mathcal{Q}_l \to \mathcal{Q}_s$ by simply extracting the subword of length $s$ in the decomposition above. The only possibilities for ambiguity occur if $\theta-1$ or $m-1$ lie within  $S$, in which case we specify that we choose the leftmost subword of the appropriate size. For a binary word $z \in \mathcal{Q}_l$ we set $z^{(s)} = \pi^{(s)}(z)$,  and $z^\prime$, $z^{\prime \prime}$  to be the projections onto the penultimate and final subwords, respectively. In a mild abuse of notation, we shall sometimes use the same notation $w^{(s)} , w^\prime$ for projections from words of the length $l-m$ onto the appropriate subwords of length $s$.

\begin{defn}
A subset of a hypercube is \textit{$k$-separated} if each pair of distinct vertices within the subset are a distance of at least $k$ apart. If $\Delta$ is a minimal dominating set for $\mathcal{Q}_l$ and $S$ is an $m$-decomposition of $l$, then a subset $E\subseteq \Delta$ is $(S,m)$-admissible if $E^{(s)} \coloneq \pi^{(s)}(E) \subseteq \mathcal{Q}_s$ is $3$-separated for each $s \in S$.
\end{defn}

With all this in place, we can present our main result,  with the caveat that converting this into useful bounds will take considerable further work.


\begin{thm}\label{Smain}
Fix $l>1$ and choose $\Delta$ to be a minimal dominating set for $\mathcal{Q}_l$. Set $(\theta,\xi)= \Theta(l+1)$. Suppose $S$ is an $m$-decomposition of $l$  for some $m \geq 0$ with the property that $m < \theta - |S|$.

Then for any $(S,m)$-admissible subset $E \subseteq \Delta$,
\[\gamma_{m+\theta+l} \leq 2^{m+\theta} \gamma_l -  2^m |E| .\] 
If $\widecheck{l}+m + \theta < 2^{\widehat{l}}$, this translates to 
\[ \chi_{m+\theta+l}  \geq \chi_{l} +2^{\widehat{l} -l - \theta} |E|  .\]
\end{thm}

Before we begin the proof, we remark that, at first glance, this estimate does not appear to improve upon the natural bound propagation for $m>0$. However, the fact that $S$ is an $m$-decomposition and $E$ is $(S,m)$-admissible results (for some ranges of $m$) in growth in $|E|$ as $m$ increases. 

\begin{proof}

To exploit \eqref{decom}, we consider maps $e_s \colon \mathcal{Q}_s \to \mathcal{Q}_{l-m}$ for $s \in S$ defined by 
\[    \pi^\prime \circ  e_s =  0\cdots 0,   \quad    \pi^{\tilde{s}} \circ e_s    =  \begin{cases}  \text{id} , \quad s =\tilde{s}, \\ 0 \cdots 0, \quad s \ne \tilde{s}. \end{cases}\]  We also introduce isometric embeddings $\iota^j_i\colon \mathcal{Q}_i \to \mathcal{Q}_j$ for $1 \leq i <j$ by  $\iota^j_i (y) = y1 0\cdots 0$. Additionally, we employ the convention that $\mathcal{Q}_{-1} = \mathcal{Q}_{0} =\{ \circ \}$ with $\circ$ denoting the empty word, and set  $\iota^j_{-1}(\circ ) = 0\cdots 0$ and $\iota^j_0 (\circ ) = 10\cdots 0$. 

 With this in hand, we then decompose $\mathcal{Q}_\theta$ as
\begin{equation}\label{iota} \mathcal{Q}_\theta \cong  \bigcup\limits_{k =-1}^{\theta-1} \iota^\theta_k (\mathcal{Q}_k),\end{equation}
This decomposition allows us to define several useful maps. First, we construct $\rho  \colon \mathcal{Q}_{\theta} \times \mathcal{Q}_{l-m} \to \mathcal{Q}_{\theta+l-m}$  by the rule
\[ \rho (\iota^\theta_k y , z ) =  \begin{cases}  \iota^\theta_k y \cdot z, &\text{ if $k \notin S$,}\\
 \iota^\theta_sz^{(s)}  \cdot    (z  + e_sz^{(s)}  + e_s y), & \text{ if $k=s \in S$}\end{cases}\]
 In other words, $\rho$ identifies which piece of the  decomposition \eqref{iota} its input lies in and then swaps entries to the left of the final $1$ with the appropriate subword $z^{(s)}$ in the concatenation $xz$. 
 



Our next ingredient is yet another map, this time $t \colon \mathcal{Q}_\theta \to \mathcal{Q}_{l-m}$, defined as follows:
\begin{itemize}
\item For $k \notin S$, 
\[ t (\iota^\theta_k y ) := \begin{cases}  \underbrace{0 \cdots 0}_{\Sigma} \cdot \iota_k^\theta y, &\text{ if $k\geq 1$} \\ 0 \cdots 0, \quad  &\text{ $k=-1,0$.}  \end{cases}\]
\item For $s \in S$ and $y \in \mathcal{Q}_s$, there are two cases. If there exists $\sigma \in E$ such that $ \| y -\sigma^{(s)} \| \leq 1$, then
\[ t( \iota_s y) := e_s (\sigma^{(s)} ) +  \begin{cases}    \delta^{l-m}_{\Sigma + s+1}, &\text{ if $y=\sigma$, } \\ 0 \cdots 0, \quad  &\text{ otherwise.}  \end{cases}\]
If no such $\sigma$ exists, then $t(\iota_sy) = 0 \cdots 0$.
\end{itemize}
For the second part, note that the $3$-separation property for $E$ ensures that $\sigma$ is unique if it exists.
Next, we combine $\rho$ and $t$ into a map $f \colon \mathcal{Q}_\theta \times \mathcal{Q}_{l-m} \to \mathcal{Q}_{\theta+l-m}$ by 
\[   f (x , z)  = \rho \Big(  x  \cdot  \big(z-t(x) \big) \Big).\] 
This function $f$ has the useful, easily checked, property that
\begin{equation}\label{f}
\|f( x ,z) - f(x, w) \|= \|z-w\|.
\end{equation}

Unfortunately, to account for the parameter $m$,  we need yet another construction. Set $S^c = \{1,\dots,\theta-1\} \setminus S$. The assumption $m < \theta- |S|$ implies there exists an injective map $ a\colon \{1,\dots,m\} \to S^c$. This induces a linear map $ a\colon \mathcal{Q}_ m \to \mathcal{Q}_\theta$ as follows: set 
\[ a(\delta^m_i) := \iota_{a(i)}(0\cdots 0) = \delta^\theta_{a(i)+1}, \quad  1 \leq i \leq m, \]
 then extend $a$ by $\mathbb{F}_2$-linearity.   Note that the definition of $a$ implies that if $a(\alpha) =\iota_k y$ then $k \notin S$.  In particular $a (\delta^m_i) = \iota_{a(i)} (0\cdots 0)$ and $t (a (\delta^m_i)) = \delta^{l-m}_{\Sigma+ a(i)+1}$. Furthermore, while $t$ is not linear in general, it is linear when restricted to $\text{Im}(a)$. 

At long last we are finally ready to construct some dominating sets. Just kidding! We first introduce one more map $g \colon \mathcal{Q}_m \times \mathcal{Q}_\theta \times \mathcal{Q}_l \to \mathcal{Q}_{m+\theta+l}$ by
\[ g(\alpha, x, z) =  \alpha \cdot \left[ a(\alpha) \cdot t(a(\alpha))  + f ( x - a(\alpha), \hat{z}  ) \right]  \cdot (z^{\prime \prime} +\alpha) .\] 
With this, we can finally define 
\[ D := g( \mathcal{Q}_m, \mathcal{Q}_\theta, \Delta), \quad  V = \{  g(\alpha, a(\alpha), \sigma) \big|  \alpha \in \mathcal{Q}_m,  \sigma \in E\} \subset D.\]
and make two claims:
\begin{enumerate}[(A)]
\item $D$ is a dominating set for $\mathcal{Q}_{m+\theta+l}$.
\item $\mathcal{D}( V) \subseteq \mathcal{D}(D \setminus V)$.
\end{enumerate}
We note that $|D| = 2^{m+\theta} \gamma_l$ and $|V| = 2^m |E|$. So if these claims are true, the result follows immediately.

Part (A) is essentially an exercise in unwinding  definitions. First, it is straight forward to check that $g$ is injective, and hence bijective. Next, for fixed $\alpha \in \mathcal{Q}_m$ and $x \in \mathcal{Q}_\theta$, using \eqref{f} we have
\begin{align*} \| &g(\alpha,x,z) - g(\alpha,x,w)\|  =  \| f(x -a(\alpha), 
\hat{z}) -  f(x -a(\alpha), \hat{w})\|  + \|z^{\prime \prime} + w^{\prime \prime}\| \\
&= \|  \hat{z} - \hat{w}  \| + \|z^{\prime \prime} - w^{\prime \prime}\| = \|z - w\|,
\end{align*}
where $\hat{z}$ is the projection of $z \in Q_{l}$ onto $Q_{l-m}$ given by lopping off the final subword $z^{\prime \prime}$. Thus $g(\alpha,x, \Delta)$ dominates $g(\alpha,x,\mathcal{Q}_l)$. Claim (A) then follows from bijectivity of $g$.


Establishing (B) requires a series of computations. Fix $\alpha \in \mathcal{Q}_m$ and $\sigma \in E$, then set 
\[ u = g(\alpha, a(\alpha), \sigma)= \alpha \cdot \left[ a(\alpha) \cdot t(a(\alpha))  + 0\cdots 0 \cdot \hat{\sigma}  ) \right]  \cdot (\sigma^{\prime \prime} +\alpha).\]
First we note  $0\notin S$, so 
\begin{align*}
D \setminus V & \ni  g(\alpha, a(\alpha) +\delta^\theta_1 , \sigma)\\
 &=  \alpha \cdot \left[ a(\alpha) \cdot t(a(\alpha))  + f (\iota_0(\circ) , \hat{\sigma}  ) \right]    \cdot (\sigma^{\prime \prime} +\alpha) = u + \delta^{m+\theta+l}_{m+1}.
 \end{align*}
 Similarly, for $k \notin S$, $k>0$,
 \begin{align*}
D \setminus V & \ni  g(\alpha, a(\alpha) +\delta^\theta_{k+1} , \sigma)\\
 &=  \alpha \cdot \left[ a(\alpha) \cdot t(a(\alpha))  + f (\iota_{k}(0\cdots 0) , \hat{\sigma}  ) \right]    \cdot (\sigma^{\prime \prime} +\alpha)  \\ &= u +(0\cdots 0)  \cdot \iota_{k}(0\cdots 0) \cdot   t(\iota_{k}(0\cdots 0))  \cdot (0\cdots 0) \\ &= u + \delta^{m+\theta+l}_{m+k+1} + \delta^{m+\theta+l}_{m+\Sigma+k+1} 
 \end{align*}
Now for $s \in S$,
\begin{align*}
D &\setminus V\ni  g(\alpha , a(\alpha) + \iota^\theta_s  \sigma^{(s)}  ,\sigma)  =  \alpha \cdot \left[ a(\alpha) \cdot t(a(\alpha))  + f (\iota^\theta_s \sigma^{(s)} , \hat{\sigma}  ) \right]    \cdot (\sigma^{\prime \prime} +\alpha) \\
&=  u +  (0\cdots 0) \cdot  \big[ \iota^\theta_s (0\cdots 0) \cdot (\hat{\sigma} + \delta^{l-m}_{\Sigma+s+1} )\big] \cdot (0\cdots 0)  \\
&= u + \delta^{m+\theta+l}_{m+s+1} + \delta^{m+\theta+l}_{m+\theta+\Sigma+s+1} 
\end{align*}
and for $1 \leq j \leq s \in S $
\begin{align*}
D &\setminus V\ni  g(\alpha , a(\alpha) + \iota^\theta_s  \sigma^{(s)} + \delta^s_j ,\sigma)  =  \alpha \cdot \left[ a(\alpha) \cdot t(a(\alpha))  + f (\iota^\theta_s \sigma^{(s)} + \delta^s_j  , \hat{\sigma}  ) \right]    \cdot (\sigma^{\prime \prime} +\alpha) \\
&=  u +  (0\cdots 0) \cdot  \big[ \iota^\theta_s (0\cdots 0) \cdot (\hat{\sigma} + e_s \delta^s_j )\big] \cdot (0\cdots 0) \\
& = u + \delta^{m+\theta+l}_{m+s+1} +  \underbrace{0 \cdots 0}_{m + \theta} \cdot (e_s \delta^s_j) \cdot (0\cdots 0) \
 \end{align*}
Furthermore, for $1 \leq i \leq m$, 
\begin{align*}
D &\setminus V  \ni  g( \alpha+ \delta^m_i , a(\alpha), \sigma) \\&=   (\alpha+ \delta^m_i ) \cdot  \left[ a(\alpha+ \delta^m_i ) \cdot t(a(\alpha+ \delta^m_i ))  + f ( \iota^\theta_{a(i)}(0\cdots 0), \hat{z}  ) \right]   \cdot (z^{\prime \prime} +\alpha+ \delta^m_i )\\
&= (\alpha+ \delta^m_i )\cdot \Big[ a(\alpha) \cdot t(a(\alpha))  + \delta^\theta_{a(i)+1} \cdot  \delta^{l-m}_{\Sigma+a(i)+1}  \\
& \quad  + \delta^\theta_{a(i)+1} \cdot ( \hat{z}  +\delta^{l-m}_{\Sigma+a(i)+1} ) \Big]   \cdot (z^{\prime \prime} +\alpha+ \delta^m_i ) = u + \delta^{m+\theta+l}_i + \delta^{m+\theta+l}_{\theta+l+i}
\end{align*}

A careful analysis, considering the full range of every free parameter, shows that every element of $\mathcal{D}(u)$ is dominated by some element of $D\setminus V$. Thus $B$ holds.

\end{proof}

\section{New Upper Bounds on $\chi$}

We shall apply Theorem \ref{Smain} in the case $l = 2^{\widehat{n}} -1$ and $(\theta,\xi,\psi,\xi^\prime)  = \Theta^{(2)}(2^{\widehat{n}})$. For each $m \geq 0$, the goal is then to maximize $|E|$ over all minimal dominating sets $\Delta$ for $\mathcal{Q}_{2^{\widehat{n}} -1}$,  $m-$decompositions $S$ of $2^{\widehat{n}} -1$ and $S$-admissible subsets $E$.  The reader would be forgiven for wondering if we have translated a simply stated, intractable optimization problem for a more complicated one. However, we can extract improved bounds for a wide range of $\gamma_n$.

However, Corollary \ref{graph} comes to the rescue and makes the problem accessible. For $l = 2^{\widehat{n}} -1$ we can choose our minimal dominating set in the form $\Delta = \{ (x,h(x) ) \big| x \in \mathcal{Q}_{2^{\widehat{n}} -1 -\widehat{n}} \}$ where $h$ is a $\mathbb{F}_2$-linear map $ \mathcal{Q}_{2^{\widehat{n}} -1 -\widehat{n}} \to \mathcal{Q}_{\widehat{n}}$. In this case, we have free choice for the first $2^{\widehat{n}} -1 -\widehat{n}$ entries for $x \in \Delta$, however the final $\widehat{n}$ positions are then uniquely determined. For $\widehat{n} >2$, $\theta- 1> \widehat{n} $, thus the letters that determine whether our $S$-admissible sets are $3$-separable can be freely chosen.

If we choose an $m$-decomposition $S \subset \{1,\dots,\theta-1\}$ of $2^{\widehat{n}}-1$ such that $\theta-|S|>m$, then $E$ comes from the split
\[ \underbrace{\star \cdots \star}_\Sigma \underbrace{\star \cdots \star}_{\theta-1} \underbrace{\star \cdots \star}_{m}.\]
For $\widehat{n}>2$,  $\theta-1 > \widehat{n}$. The maximal size of $E$ can thus be determined as
\[   |E| =  2^{\theta-1+m - \widehat{n}}   \prod\limits_{ s \in S} \lambda(s),\]
where $\lambda(s)$ is the maximal size of $3$-separated set in $\mathcal{Q}_s$. This reduces the problem to maximizing  $\prod_{ s \in S} \lambda(s)$ over sets $\{ 1 \leq s_1 < \cdots < s_q \leq \theta-1\}$ such that $\sum s_i = 2^{\widehat{n}} - \theta - m$.

Note, there is a subtle interaction with $m$ here. Increasing $m$ by $1$ doubles the multiplicative factor outside the product but reduces either the number of terms or a value of $s$ inside the product. In many cases, but certainly not all, these effects will cancel out. However, in others, the reduction in $S$ does not change the value of the product. The upshot of this is the following:

\begin{thm}\label{chimain}
For $n = 2^{\widehat{n}}-1 +\theta$,  where $(\theta,\xi) =\Theta(2^{\widehat{n}})$,  $m \geq 0$, and any $m$-decomposition $S$ of $2^{\widehat{n}}-1$,
\begin{equation}\label{chibound}
\chi_{n+m}  \geq 2^{m - 2^{\widehat{n}}}   \prod\limits_{ s \in S} \lambda(s)
\end{equation}
\end{thm}
The goal is thus to optimize \eqref{chibound} over $m$-decompositions $S$. Unfortunately, $\lambda(k)$ is, in general, difficult to compute. However, it is straightforward to produce lower bounds, which can be translated into lower bounds for $\chi_n$.

\begin{lemma}\label{lambdabound}
For $n \geq 1$,
\[\lambda(n) \geq 2^{n -\widehat{n} - \one(\widecheck{n} >0)}.\]
If $\widecheck{n} =0$, then equality holds.
\end{lemma}

\begin{proof}
The case $\widecheck{n}=0$ is obvious as the minimal dominating sets produced using Lemma \ref{propagate} are $3$-separating. The case $\widecheck{n}>0$ follows immediately from the identity
\begin{equation}\label{lambdaineq} \lambda(p+q)\geq 2^{q-1} \lambda(p)  \end{equation}
whenever $1 \leq q \leq p$. To check this, suppose $D$ is a $3$-separated subset of $\mathcal{Q}_p$ with $|D| = \lambda(p)$ and define a subset of $\mathcal{Q}_{p+q}$ by
\[ \tilde{D} = \{  ( x, x+ y) \big| y \in D,  x \in\mathcal{Q}_q, \pi(x)=0 \}.\]
Here, we are implicitly embedding $\mathcal{Q}_q$ into $\mathcal{Q}_p$ by $ x \mapsto x \cdot 0\cdots 0$. We claim $\tilde{D}$ is also 3-separated. For $z = (x,x+y), w =(u,u+v) \in \tilde{D}$ then 
\[ \|z+w \| = \| x+u\| + \| x + y+u+v\|\]
If $x \ne u$ , then $\|x +u\| \geq 2$ as they are both even. But since $y,v \in D$ they are either equal or distance $\geq 3$ apart. If they are equal then clearly the second term has magnitude $\geq 2$. If they aren't equal,  the only way the second term can vanish is if $\|x+u\| \geq 3$.  In any combination, $\|z+w\| \geq 3$.

If $x =u$, then $\|z+w\| = \|y+v\|$. But then either $y=v$ or $\|y+v\|\geq 3$, which completes the proof.

\end{proof}


The guiding philosophy is thus that an optimal $m$-decomposition $S$ will favor larger values of $s$ over smaller.

We are now ready to produce some bounds. Instead of producing a cumbersome general formula, we shall illustrate using specific cases. The calculations are all elementary and the method is easily extendible to any $\widehat{n}>0$. Our first step, however, is the following easily checked observation. 

\begin{lemma}\label{decoms}
Let $(\theta,\xi,\psi,\xi^\prime)  = \Theta^{(2)}(2^{\widehat{n}})$. 
\begin{itemize}
\item For  $0 \leq m < \xi^\prime$,  set \[S_m  = \{ \xi^\prime-m\} \cup \{ \psi+1,\psi+2,\dots, \theta-1\}.\] 
\item For $ \xi^\prime  \leq m \leq \psi$, set
\[ S_m =  \{ \psi+1 - m +\xi^\prime\} \cup \{ \psi+2,\dots, \theta-1\}.\]
\end{itemize}
Then $S_m$ is an $m$-decomposition of $2^{\widehat{n}}-1$ satisfying $ m < \theta - \left| S_m \right|$.
\end{lemma}

 Throughout the examples below, we shall use the decompositions from this lemma and the lower bounds on $\lambda$ from Lemma \ref{lambdabound} as a stand in for the actual values of $\lambda$ when calculating the product in Theorem \ref{chimain}. All this means is that the bounds obtained might be (modestly) improved if and when we refine our lower bounds on $\lambda$ or calculate it exactly. The reader is cautioned that the first few cases are included to aid familiarity with the method while the calculations can be still done easily in your head. For small $\widehat{n}$, our bounds fail to improve on those already known.

\subsection{ $\widehat{n} =3$}

Compute $\Theta^{(2)}(8) = (4,2,2,1)$. Then from the observation above, we can set $S_0 =\{1,3\}$, $S_1 =\{3\}$, $S_2=\{2\}$. All  $m\geq 3$ will automatically fail the condition $m<\theta-|S|$ and hence are considered out of range. From Theorem \ref{chimain},   we see
\[ \chi_{11} \geq 2^{-7}, \quad \chi_{12} \geq 2^{-6}.\]
We see no further improvement at $m=2$ as $\lambda(3)=2$ but $\lambda(2)=1$, thus $\prod_{S_m} \lambda(s)$ drops by $1/2$ at this level, canceling any gain. Thus this method yields $\chi_{12}  \geq 2^{-6}$.

\subsection{ $\widehat{n} =4$}

Compute $\Theta^{(2)}(16) = (6,5,3,1)$. Using Lemma \ref{decoms}, we choose $S_0 =\{1,4,5\}$, $S_1 =\{4,5\}$,  $S_2=\{3,5\}$,  and $S_3=\{2,5\}$. All $m\geq4$ are out of range. From Theorem \ref{chimain},  we see
\[ \chi_{21} \geq 2^{-13}, \quad \chi_{22} \geq 2^{-12},\quad \chi_{23} \geq 2^{-11}.\]
$\lambda(4)=\lambda(3)=2$, so we do obtain a doubling at the $m=2$ stage. However $\lambda(2)=1$, so we see no further improvement at $m=3$..

\subsection{ $\widehat{n} =5$}

Compute $\Theta^{(2)}(32) = (8,4,3,2)$. Using Lemma \ref{decoms}, we choose  $S_0 =\{2,4,5,6,7\}$, $S_1 =\{1,4,5,6,7\}$,  $S_2=\{4,5,6,7\}$, and $S_3=\{3,5,6,7\}$. All $m\geq4$ are out of range. From Theorem \ref{chimain},  we see
\[ \chi_{39} \geq 2^{-22}, \quad \chi_{40} \geq 2^{-21}, \quad \chi_{41} \geq 2^{-20}, \quad \chi_{42} \geq 2^{-19}.\]
In this case $\prod_{S_m} \lambda(s)$ is constant over the range $0 \leq m \leq 3$ so we see improvements at each stage.

\subsection{ $\widehat{n} =6$}

At this stage, we finally begin to see values of $n$ for which no improvement was previously known.

Compute $\Theta^{(2)}(64) = (11,2,2,1)$. Using Lemma \ref{decoms}, we choose  $S_0 =\{1,3,\dots, 10\}$,  $S_1 =\{3,\dots,10\}$ and $S_2 =\{2,4\dots,10\}$. All $m\geq3$ are out of range. From Theorem \ref{chimain},  we see 
\[ \chi_{74} \geq 2^{-38}, \quad \chi_{75} \geq 2^{-37}.\]
As $\lambda(2) =1$, while $\lambda(3)=2$, there is no improvement at $m=2$. At the time of writing, these are the best known bounds. In fact, $\chi_{76},\dots,\chi_{78}  \geq 2^{-37}$ are also the current best. However $79 \in W(9)$ and so the known sharp bound on $n=9$ propagates up to yield $\chi_{79} \geq 2^{-5}$.

\subsection{ $\widehat{n} =7$}

Compute $\Theta^{(2)}(128) = (16,8,4,2)$. Using Lemma \ref{decoms}, we choose $S_0 =\{2,5,\dots, 15\}$, $S_1 =\{1,5,\dots,15\}$,  $S_2=\{5,\dots,15\}$, $S_3 =\{4,6,\dots,15\}$, and $S_4 =\{3,6,\dots,15\}$. All $m\geq 5$ are out of range.  From Theorem \ref{chimain},  we see
\[ \chi_{143} \geq 2^{-59}, \quad \chi_{144} \geq 2^{-58}, \quad \chi_{145} \geq 2^{-57}, \quad \chi_{147} \geq 2^{-56}.\]
$\lambda(5) = 4 = 2 \lambda(4)$, so we see no improvement at the $m=3$ stage. The best bound on $\chi_{146} \geq 2^{-57}$ which it inherits from $n=145$. However $\lambda(3)=\lambda(4)=2$, so we do see an improvement again at $m=4$.  Again, these are the best known bounds at these levels. We also note that $149 \in W(74)$ and so our bounds at the previous level propagate up: $\chi_{149}, \chi_{150} \geq 2^{-38}$, and $\chi_{151},\dots,\chi_{158} \geq 2^{-37}$.

\subsection{ $\widehat{n} =8$}

Compute $\Theta^{(2)}(256) = (23,20,6,1)$. Using Lemma \ref{decoms}, we choose  $S_0 =\{1,7\dots, 22\}$, $S_1 =\{7,\dots,22\}$,  $S_2=\{6,8,\dots,22\}$, $S_3 =\{5,8,\dots,22\}$, $S_4 =\{4,8,\dots,22\}$, $S_5 =\{3,8,\dots,22\}$,  and $S_6 =\{2,8,\dots,22\}$. All $m\geq 7$ are out of range.   From Theorem \ref{chimain},  we see
\[ \chi_{278} \geq 2^{-94}, \quad \chi_{279} \geq 2^{-93}, \quad \chi_{283} \geq 2^{-92}.\]

For $m=2,\dots, 4$, we simply inherit the $2^{-93}$ bound from $n=279$. However, for $m=5$, we see an improvement again, but there is no improvement at $m=6$.

\subsection{ $\widehat{n} >8$} For any fixed $\widehat{n}$ the computations involved are elementary and  require nothing computationally intensive. Indeed, bounds could be calculated with pen and paper significantly past the point where we have stopped. Although a general formula is fiddly to the point of being unworkable, this algorithm would be straightforward to code.

A couple of remarks are warranted. First, these bound improvements are significant for large $n$. For example, $\gamma_{278} \leq  2^{255} - 2^{161}$,  an improvement on the order of $10^{48}$ over the generic bound. However,  it does have to be said that these bound improvements are small in a relative sense. They are unlikely to be sharp, even in the cases where they are the best known. The values of $\chi_n$ that this theorem produces rapidly approach $0$ as $\widehat{n} \to \infty$.  If $n$ can be shown to lie inside $W(p)$  for some integer $p$ where an alternate bound improvement is known, then these methods are unlikely to produce a superior result. However, it is straightforward to see that, asymptotically,  $\theta  =\Omega (2^{\widehat{n}/2})$.  The consequence of this is that  this technique applies to an infinite number of $n$ lying outside the domination wedge $W(p)$ for any fixed $p$.  Thus, no matter how many individual bounds can be numerically computed for fixed values of $n$ via simulated annealing or other computer-assisted techniques, this methodology would still produce the best known bounds in infinitely many cases.   Prior to this paper, every improved bound that the authors are aware of occurred within $W(9)$.  Therefore, our methods can be used to find the best-known explicit upper bounds for any $n$ within the range $ \theta({\widehat{n}}) \leq {\widecheck{n}} < 2^{{\widehat{n}} -2}$. Figure \ref{best} illustrates this phenomenon. The new improved bounds appear in a 'curved wedge' lying above those previous established. The height of this wedge grows exponentially with $\widehat{n}$.

Finally, a minimal dominating set for $\mathcal{Q}_9$ was found in \cite{Wille96} by a computationally intensive simulated annealing technique. The set, explicitly written out in that paper,  can only be described as wild. There are few discernible, consistent patterns and the set possesses odd groupings of adjacent vertices. It feels less like something that can be built from an intuitive construction and more like the result of vertices being eliminated after a twisting of a more structured set, in the manner of Theorem \ref{Smain}. However, removing the condition $\widecheck{n} \geq \theta$ would require a more sophisticated construction than the one presented here.

\begin{figure}
\caption{The best known lower bounds on $\chi_n$ for  $n \leq 511$.  Bounds of the form $2^{-a}$ are new in this paper. Citations for all others can be found earlier in the text. Upper bounds on $\gamma_n$ can be found from the identity $\gamma_n \leq 2^{n- \widehat{n}}  (1 - \chi_n)$.} 
\label{best}
\begin{center}
\bgroup
\scriptsize
\setlength{\tabcolsep}{0.6em}
\begin{tabular}{|cc|cccccccc|}
\hline 
 & & \multicolumn{8}{c|}{${\widehat{n}}$} \\
& & 1  & 2 & 3 &4 & 5 & 6 & 7 &8 \\
\hline
& 0 &0&0& 0& 0& 0& 0&0& 0  \\
& 1& 0 & 0 & 0& 0& 0& 0& 0 & 0\\
& 2& \mynull & 1/8 & 1/32 &0 &0 & 0& 0  &0\\
& 3& \mynull &  1/4 & 1/16 &0 &0 &0& 0 &0  \\
& 4& \mynull & \mynull &  1/4 &1/32&0 & 0& 0 &0\\
& 5& \mynull & \mynull & 33/128 &1/32&0& 0& 0&0\\
& 6--7& \mynull & \mynull & 5/16 &1/16&0&0& 0&0 \\
& 8--9& \mynull & \mynull & \mynull &1/4& 1/32&0& 0&0 \\
& 10& \mynull & \mynull & \mynull &33/128&1/32&0& 0&0\\
& 11& \mynull & \mynull & \mynull &33/128&1/32& $  2^{-38} $ & 0&0\\
&12--15& \mynull & \mynull & \mynull & 5/16&1/16& $  2^{-37} $ & 0&0\\
& 16& \mynull & \mynull & \mynull & \mynull &1/4& 1/32 & $2^{-69}$& 0 \\
& 17& \mynull & \mynull & \mynull & \mynull &1/4& 1/32 & $2^{-68}$& 0 \\
& 18-19& \mynull & \mynull & \mynull & \mynull &1/4&1/32  & $2^{-67}$& 0 \\
 & 20--21& \mynull & \mynull & \mynull & \mynull & 33/128& 1/32 & $2^{-66}$& 0 \\
& 22& \mynull & \mynull & \mynull & \mynull & 33/128&1/32  & $  2^{-38} $& 0 \\
& 23& \mynull & \mynull & \mynull & \mynull & 33/128&1/32  & $  2^{-38} $& $2^{-93} $\\
& 24& \mynull & \mynull & \mynull & \mynull &5/16& 1/16&  $  2^{-37} $& $2^{-94}$\\
& 25--27& \mynull & \mynull & \mynull & \mynull &711/2048& 1/16&  $  2^{-37} $& $2^{-94}$\\
& 28-29& \mynull & \mynull & \mynull & \mynull &711/2048&1/16 & $  2^{-37} $& $2^{-91}$\\
& 30--31& \mynull & \mynull & \mynull & \mynull &3/8& 1/16& $  2^{-37} $& $2^{-91}$\\
${\widecheck{n}}$  & 32--33& \mynull & \mynull & \mynull & \mynull & \mynull & 1/4 & 1/32& $2^{-69}$\\
& 34--35& \mynull & \mynull & \mynull & \mynull & \mynull & 1/4 & 1/32& $2^{-68}$ \\
& 36--39& \mynull & \mynull & \mynull & \mynull & \mynull & 1/4 & 1/32& $2^{-67}$ \\
& 40--43& \mynull & \mynull & \mynull & \mynull & \mynull & 33/128& 1/32& $2^{-66}$\\
& 44-47& \mynull & \mynull & \mynull & \mynull & \mynull &  33/128& 1/32  & $  2^{-38} $\\
& 48--49& \mynull & \mynull & \mynull & \mynull & \mynull & 5/16&  1/16& $2^{-37}$\\
& 50--59& \mynull & \mynull & \mynull & \mynull & \mynull & 711/2048&  1/16& $2^{-37}$ \\
& 60--63& \mynull & \mynull & \mynull & \mynull & \mynull &3/8&  1/16 & $2^{-37}$\\
& 64--79& \mynull & \mynull & \mynull & \mynull & \mynull & \mynull & 1/4 & 1/32\\
& 80--95& \mynull & \mynull & \mynull & \mynull & \mynull & \mynull & 33/128 & 1/32\\
& 96--99& \mynull & \mynull & \mynull & \mynull & \mynull & \mynull & 5/16 & 1/16\\
& 100--119& \mynull & \mynull & \mynull & \mynull & \mynull & \mynull & 711/2048 & 1/16\\
& 120--127& \mynull & \mynull & \mynull & \mynull & \mynull & \mynull & 3/8 & 1/16\\
& 128--159& \mynull & \mynull & \mynull & \mynull & \mynull & \mynull & \mynull & 1/4\\
&160--187& \mynull & \mynull & \mynull & \mynull & \mynull & \mynull & \mynull & 33/128 \\
& 192-199& \mynull & \mynull & \mynull & \mynull & \mynull & \mynull & \mynull & 5/16 \\
& 200--239& \mynull & \mynull & \mynull & \mynull & \mynull & \mynull & \mynull & 711/2048 \\
& 240--255& \mynull & \mynull & \mynull & \mynull & \mynull & \mynull & \mynull & 3/8 \\
\hline
\end{tabular}
\egroup
\end{center}
\end{figure}

  \bibliography{Domination}{}

\begin{thebibliography}{10}
\expandafter\ifx\csname url\endcsname\relax
  \def\url#1{\texttt{#1}}\fi
\expandafter\ifx\csname urlprefix\endcsname\relax\def\urlprefix{URL }\fi
\expandafter\ifx\csname href\endcsname\relax
  \def\href#1#2{#2} \def\path#1{#1}\fi

\bibitem{Wille96}
L.~Wille, New binary covering codes obtained by simulated annealing, IEEE
  Trans. Inform. Theory 42 (1996) 300--302.

\bibitem{MR1873942}
P.~R.~J. \"{O}sterg\aa rd, U.~Blass,
  On the size of optimal binary codes
  of length 9 and covering radius 1, IEEE Trans. Inform. Theory 47~(6) (2001)
  2556--2557.


\bibitem{MR945316}
G.~J.~M. van Wee, Improved sphere bounds
  on the covering radius of codes, IEEE Trans. Inform. Theory 34~(2) (1988)
  237--245.


\bibitem{mallard}
M.~Mallard, Les invariants du n-cube, Ph.D. thesis, Univ. de Grenoble, France
  (1981).

\bibitem{Keri}
G.~Kéri, Tables for bounds
  on covering codes,
\newline\urlprefix\url{https://doi.org/10.1016/j.disc.2023.113752}

\bibitem{MR4658424}
Y.-S. Wu, J.-Y. Chen,
Improved lower bounds on
  the domination number of hypercubes and binary codes with covering radius
  one, Discrete Math. 347~(2) (2024) Paper No. 113752, 9.


\bibitem{MR1041836}
L.~T. Wille, Improved binary code coverings by simulated annealing, in:
  Proceedings of the {T}wentieth {S}outheastern {C}onference on
  {C}ombinatorics, {G}raph {T}heory, and {C}omputing ({B}oca {R}aton, {FL},
  1989), Vol.~73, 1990, pp. 53--58.

\bibitem{MR859092}
G.~D. Cohen, A.-C. Lobstein, N.~J.~A. Sloane, Further results on the
  covering radius of codes, IEEE Trans. Inform. Theory 32~(5) (1986) 680--694.


\bibitem{MR1669523}
P.~R.~J. \"{O}sterg\aa rd, W.~D. Weakley, Constructing covering codes
  with given automorphisms, Des. Codes Cryptogr. 16~(1) (1999) 65--73.

\bibitem{MR1483747}
P.~R.~J. \"{O}sterg\aa rd, M.~K. Kaikkonen, New upper bounds for
  binary covering codes, Discrete Math. 178~(1-3) (1998) 165--179.


\bibitem{MR290860}
J.~G. Kalbfleisch, R.~G. Stanton, J.~D. Horton, On covering sets and
  error-correcting codes, J. Combinatorial Theory Ser. A 11 (1971) 233--250.


\bibitem{MR1477282}
L.~Habsieger, \href{https://doi.org/10.1016/S0012-365X(96)00290-7}{Binary codes
  with covering radius one: some new lower bounds}, Discrete Math. 176~(1-3)
  (1997) 115--130.


\end{thebibliography}
\bibliographystyle{elsarticle-num} 
\end{document}